\documentclass[11pt]{article}
\usepackage{amsmath, amsthm, amssymb, latexsym}
\usepackage{graphicx}
\usepackage{amsmath}
\usepackage{amsfonts}
\usepackage{enumerate}
\usepackage{latexsym}
\usepackage{epsfig}
\usepackage{amsthm}  
\usepackage{amsmath}
\usepackage{amssymb} 
\usepackage{latexsym}
\usepackage{epsfig} 
\usepackage{amssymb,latexsym}
\usepackage[all]{xy}

\usepackage{a4,fullpage,amssymb,epsf,psfrag,times}

\newcommand     {\comment}[1]   {}
\newcommand{\mute}[2] {}
\newcommand     {\printname}[1] {}

\newtheorem{theorem}{Theorem}[section]
\newtheorem{proposition}[theorem]{Proposition}

\newtheorem{lemma}[theorem]{Lemma}
\newtheorem{claim}[theorem]{Claim}

\newtheorem{corollary}[theorem]{Corollary}

\newcommand{\id}{\mathop{{\rm Id}}\nolimits}

\newcommand{\pt}{\mathop{{\rm pt}}\nolimits}
\newcommand{\PSL}{\mathop{{\rm PSL}}\nolimits}
\newcommand{\SO}{\mathop{{\rm SO}}\nolimits}
\newcommand{\Aut}{\mathop{{\rm Aut}}\nolimits}

\newcommand{\reg}{\mathop{{\rm reg}}\nolimits}
\newcommand{\R}{{\mathbb R}}
\newcommand{\E}{{\mathcal E}}

\newcommand{\fB}{{\mathcal B}}
\newcommand{\mS}{{\mathcal S}}
\newcommand{\J}{{\mathcal J}}
\newcommand{\M}{{\mathcal M}}
\newcommand{\W}{{\mathcal W}}
\newcommand{\C}{{\mathbb C}}
\newcommand{\CP}{{\mathbb C}{\mathbb P}}
\newcommand{\Z}{{\mathbb Z}}

\newcommand{\op}[1]{\!\!\mathop{\rm ~#1}\nolimits}

\newcommand{\scriptop}[1]{\!\!\mathop{\mbox{\rm \scriptsize ~#1}}\nolimits}

\newenvironment{remark}{\refstepcounter{theorem}\par\medskip\noindent{\bf
Remark~\thetheorem~~}}{\unskip\nobreak\hfill\hbox{ $\oslash$}\par\bigskip}

\newenvironment{noTitle}{\refstepcounter{theorem}\par\medskip\noindent{\thetheorem~~}}{\unskip\nobreak\hfill\hbox{ $\oslash$}\par\bigskip}

\newenvironment{example}{\refstepcounter{theorem}\par\medskip\noindent{\bf
Example~\thetheorem~~}}{\unskip\nobreak\hfill\hbox{ $\oslash$}\par\bigskip}

\newenvironment{definition}{\refstepcounter{theorem}\par\medskip\noindent{\bf
Definition~\thetheorem~~}}{\unskip\nobreak\hfill\hbox{}\par\bigskip}
\begin{document}

\title{Symplectic geometry on moduli spaces of $J$\--holomorphic curves}
\author{J. Coffey, L. Kessler, and \'A. Pelayo} 

\date{}

\maketitle

\begin{abstract}
Let $(M,\omega)$ be a symplectic manifold, and $\Sigma$ a compact Riemann surface. We define a $2$-form $\omega_{\mathcal{S}_{\op{i}}(\Sigma)}$ on the space $\mathcal{S}_{\op{i}}(\Sigma)$ of immersed symplectic surfaces in $M$, and show that the form is closed and 
non-degenerate, up to reparametrizations. Then
we give conditions on a compatible almost complex structure $J$ on $(M,\, \omega)$ that ensure that the restriction of  $\omega_{\mathcal{S}_{\op{i}}(\Sigma)}$ to the moduli space of simple immersed $J$-holomorphic $\Sigma$-curves in a homology class  $A \in \op{H}_2(M,\,\Z)$ is a symplectic form, and show applications and examples.
In particular, we deduce sufficient conditions for the existence of $J$-holomorphic $\Sigma$-curves 
in a given homology class for a generic $J$. 
\end{abstract}

\section{Introduction}

In this paper we define and study geometric structures induced on the moduli space of $J$-holomorphic curves in a symplectic manifold with a compatible almost complex structure $J$.
These constructions yield symplectic invariants of the original manifold.

Let $(M,\omega)$ be a finite-dimensional symplectic manifold, and $\Sigma$ a closed $2$-manifold. Denote by
$$\op{ev} \colon \op{C}^{\infty}(\Sigma,\, M)  \times \Sigma \to M$$ 
the \emph{evaluation map} 
$$
\op{ev}(f,\,x):= f(x).
$$

\begin{definition}
We define a $2$-form on $\op{C}^{\infty}(\Sigma,\, M)$ as the push-forward of the $4$-form ${\op{ev}}^{\ast}(\omega \wedge \omega)$ along the coordinate-projection  
$\pi_{\op{C}^{\infty}(\Sigma,\, M)} \colon \op{C}^{\infty}(\Sigma,\, M) \times \Sigma \to \op{C}^{\infty}(\Sigma,\, M)$:
  \begin{eqnarray} \label{omegas1:def}
       (\omega_{\op{C}^{\infty}(\Sigma,\, M)})_f( \tau_1, \,\tau_2):=\int_{\{f\} \times \Sigma} \iota_{( \ell_1 \wedge
    \ell_2 )} {\op{ev}}^{\ast} ( \omega \wedge \omega ).
  \end{eqnarray}
Here $\ell_{i} \in \op{T}(\op{C}^{\infty}(\Sigma,\, M) \times \Sigma)$ is a
\emph{lifting} of $\tau_{i} \in \op{T}_f(\op{C}^{\infty}(\Sigma,\, M))$,  i.e.,
$$
\op{d} ( \pi_{\op{C}^{\infty}(\Sigma,\, M)} ) {\ell_{i}}_{( f,\, x )}  = \tau_{i} \,\,\, \textup{at each point}\,\,\,
(f, \,x) \in \pi_{\op{C}^{\infty}(\Sigma,\, M)}^{- 1} (f).
$$
\end{definition}

Denote 
\begin{eqnarray} 
\mathcal{S}_{\op{i}}(\Sigma):=\{ f \colon \Sigma \to M \ \mid \ f \textup{ is an
immersion, }f^{*}\omega \textup{ is a symplectic form on }\Sigma \}. \nonumber
\end{eqnarray}
The space $\mathcal{S}_{\op{i}}(\Sigma)$ is  an open subset of the 
Fr\'echet manifold $\op{C}^{\infty}(\Sigma,\, M)$.
We identify the tangent space 
%$\op{T}_f (\mathcal{S}_{\op{i}}(\Sigma))$
with the space $\Omega^0(\Sigma,\, f^*(\op{T}\!M))$ 
of  smooth vector fields $\tau \colon \Sigma \rightarrow f^{\ast} (\op{T}\!M)$. We 
say that a vector
field  $\tau \colon \Sigma \rightarrow f^{\ast} (\op{T}\!M)$ {\em{is tangent to $f(\Sigma)$ at $x$}} if $\tau(x) \in
\op{d}\!f_{x}(\op{T}_{x}\Sigma)$. We say that $\tau$ is \emph{everywhere tangent to $f(\Sigma)$}
if $\tau$ is tangent to $f(\Sigma)$ at $x$ for
 every $x \in \Sigma$. Let $$\omega_{\mathcal{S}_{\op{i}}(\Sigma)}$$ be the $2$\--form on  $\mathcal{S}_{\op{i}}(\Sigma)$ given
by the restriction of $\omega_{\op{C}^{\infty}(\Sigma,\, M)}$.

\begin{theorem} \label{thclonon}
 The $2$-form $\omega_{\op{C}^{\infty}(\Sigma,\, M)}$ on $\op{C}^{\infty}(\Sigma,\, M)$ is well defined and closed, and $\omega_{\op{C}^{\infty}(\Sigma,\, M)} (\tau, \,\cdot )$
vanishes at $f$ if $\tau$ is  everywhere tangent to $f(\Sigma)$.

 The $2$\--form $\omega_{\mathcal{S}_{\op{i}}(\Sigma)}$ on $\mathcal{S}_{\op{i}}(\Sigma)$ is  closed, and $\omega_{\mathcal{S}_{\op{i}}(\Sigma)} (\tau, \,\cdot )$
vanishes at $f$ if and only if $\tau$ is everywhere tangent to $f(\Sigma)$ .  
\end{theorem}

Heuristically, the theorem says that the form $\omega_{\mathcal{S}_{\op{i}}(\Sigma)}$ 
descends to a non\--degenerate closed $2$\--form on the quotient space of  unparametrized 
$\Sigma$-curves. We prove the theorem in Section \ref{prop}.

\vspace{1mm}

Consider the space of $\omega$-compatible almost complex structures $\J=\J(M,\,\omega)$ on $(M,\,\omega)$.  Fix $\Sigma=(\Sigma,\, j)$, where $j$ is a complex structure on $\Sigma$. 
The moduli space 
$\M_{\scriptop{i}}(A,\,\Sigma,\,J)$ is defined as the intersection of $\mathcal{S}_{\op{i}}(\Sigma)$ with the moduli space $\M(A,\,\Sigma,\,J)$ of simple $J$-holomorphic $\Sigma$-curves in a homology class $A \in \op{H}_2(M,\,\Z)$. 
The universal moduli space  $\M_{\scriptop{i}}(A,\,\Sigma,\,\J)$ is defined as the space of pairs $\{(f,\,J) \,\, | \,\, f \in \M_{\scriptop{i}}(A,\,\Sigma,\,J)\}$.
We look at almost complex structures that are regular for the projection map $$p_{A} \colon \M(A,\,\Sigma, \,\mathcal{J}) \to \J;$$ for such a  $J$, the space $\M_{\scriptop{i}}(A,\,\Sigma,\,J)$ is a finite-dimensional manifold. Denote the set of $p_{A}$-regular $\omega$\--compatible almost complex structures by  $\mathcal{J}_{\scriptop{reg}}(A).$ This set is of the second category in $\J$. 
A class $A \in \op{H}_2(M,\, \Z)$ is \emph{$J$-indecomposable} if it does not split as a sum 
$A_1 + \ldots + A_k$ of classes all of which can be represented by non-constant $J$-holomorphic curves. 
We give the necessary background on $J$\--holomorphic curves in Section \ref{mod}.

\vspace{1mm}

If $J_* \in \mathcal{J}_{\scriptop{reg}}(A)$ is integrable, then the restriction of the form $\omega_{\mathcal{S}_{\op{i}}(\Sigma)}$ to $\M_{\scriptop{i}}(A,\,\Sigma,\,J_*)$ is non-degenerate, up to reparametrizations; see  Proposition \ref{propintsym}. Therefore under some conditions on a path starting from $J_*$, there is a neighborhood $U$ of $J_*$ in the path such that for every $J \in U$, the restriction  of the form $\omega_{\mathcal{S}_{\op{i}}(\Sigma)}$ to the finite-dimensional manifold $\M_{\scriptop{i}}(A,\,\Sigma,\,J)$ is non-degenerate up to reparametrizations; see Lemma \ref{open}. This form descends to a symplectic $2$-form 
${\widetilde{\omega}}_{\mathcal{S}_{\op{i}}(\Sigma)}$ on the quotient space  ${\widetilde{\M}}_{\scriptop{i}}(A,\,\Sigma,\,J)$ of  $\M_{\scriptop{i}}(A,\,\Sigma,\,J)$ by the proper action of the group  $\Aut(\Sigma,j)$ of bi-holomorphisms of $\Sigma$; 
see Remark \ref{remquo}. Similarly, the form $\omega_{\op{C}^{\infty}(\Sigma,\, M)}$ descends to a closed $2$-form ${\widetilde{\omega}}_{\op{C}^{\infty}(\Sigma,\, M)}$ on the quotient space ${\widetilde{\M}}(A,\,\Sigma,\,J)$ of  $\M(A,\,\Sigma,\,J)$ by the proper action of  $\Aut(\Sigma,j)$.

\vspace{1mm}

The form $ \omega_{\mathcal{S}_{\op{i}}(\Sigma)}$ yields symplectic invariants of $(M,\,\omega,\,A)$. In Section \ref{modim}, we apply Theorem \ref{thclonon}, Gromov's compactness theorem, and Stokes' theorem to deduce the following corollary.
\begin{corollary} \label{vol}
Assume that $M$ is compact.
Let  $S$ be the subset of $p_A$-regular $\omega$-compatible
$J$-s for which the class $A$ is $J$-indecomposable.  
Then 
\begin{equation} \label{int1}
\int_{\widetilde{\mathcal{M}}(A,\,\Sigma,\, J) }{\wedge^{n} {\widetilde{\omega}}_{{\op{C}}^{\infty}(\Sigma,\, M)}}
\end{equation}
is well defined and does not depend on $J \in S$. 

If there is an integrable $J_* \in S$ such that ${\mathcal{M}_{\scriptop{i}}(A,\,\Sigma,\, J_*)} \neq \emptyset$ then  
%$\int_{{\widetilde{\mathcal{M}}}_{\scriptop{i}}(A,\,\Sigma,\, J)} {\wedge^{n} {\widetilde{\omega}}_{\mS_{\scriptop{i}}(\Sigma)}} \neq 0$, hence 
$$\int_{{\widetilde{\mathcal{M}}}(A,\,\Sigma,\, J)} {\wedge^{n} {\widetilde{\omega}}_{{\op{C}}^{\infty}(\Sigma,\, M)}} \neq 0\,\,\,\,\,\textup{for every} \,\,J \in S.
$$ 
In particular, for every $J \in S$ there exists a $J$-holomorphic curve in $A$.
\end{corollary}
 When $S$ is of the second category, the corollary implies that the integral \eqref{int1} does not depend on $J$ for a generic $J$; in some cases, the corollary implies  the existence of a $J$-holomorphic $\Sigma$-curve in $A$ for a generic $J$.
In Section \ref{examples}, we give examples in which Corollary \ref{vol} and Lemma \ref{open} apply.

\vspace{1mm}

We observe that for
$J \in  \mathcal{J}_{\scriptop{reg}}(A)$ such that  ${\widetilde{\omega}}_{\mathcal{S}_{\op{i}}(\Sigma)} $  on ${\widetilde{\mathcal{M}}}_{\op{i}}(A,\,\Sigma,\, J)$ is symplectic,  we obtain
a canonical almost complex structure on  ${\widetilde{\mathcal{M}}}_{\op{i}}(A,\,\Sigma,\, J)$
 that is compatible with the form;  see Corollary \ref{coralm}. Thus this paper provides a setting to understand a symplectic manifold with a compatible almost complex structure by studying curves on the
moduli spaces of $J$-holomorphic curves, a technique that has been proven to be far reaching in algebraic
geometry.

%%%%%%%%%%%%%%%%%%%%%%%%%%%%%%%%%%%%%%%%

\section{Properties of the $2$-forms $\omega_{\op{C}^{\infty}(\Sigma,\, M)}$ and $\omega_{\mS_{\op{i}}(\Sigma)}$} \label{prop}

We first observe the following fact, that we will use through this section. 

\begin{lemma} \label{trivial:lem}
For $(\nu,\,v_{\Sigma}) \in \op{T}_{(f,\,x)}(\op{C}^{\infty}(\Sigma,\, M) \times \Sigma)$,
\begin{eqnarray}
  \op{d}(\op{ev}) _{( f, \,x )} (\nu,\,v_{\Sigma}) =  \nu_{f} (x) + \op{d}\! f_{ x } (
v_{\Sigma})  \label{De} \nonumber
\end{eqnarray}
In particular
$\op{d}(\op{ev})|_{\op{T} ( \{f\} \times \Sigma )} =\op{d} \! f$
and 
$\op{d}(\op{ev})_{(f, \, x)}(\nu,0)=\nu_f (x)$.
\end{lemma}

\subsection*{The form is well defined}

\begin{claim} \label{equi2}
Let $f \colon \Sigma \to M$
and  $\tau_1, \tau_2 \in \op{T}_{f}(\op{C}^{\infty}(\Sigma,\, M))$.
The value of  $\int_{\{f\} \times \Sigma} \iota_{( \ell_1 \wedge
    \ell_2 )} \op{ev}^{\ast} ( \omega \wedge \omega )$
 does not depend on the choice of
liftings $\ell_i$ of $\tau_i$. 
\end{claim}

\begin{proof}

Let $\ell_i, \ell'_i$ be two pairs of liftings of $\tau_i$.
 Let 
$$
{v_i}_{( f, \,x )} := {\ell_i'}_{( f, \,x )} -{\ell_i}_{( f,\,x )}.
$$  
Then
$${v_i}_{ ( f, \,x ) }= (0,\,v_{\Sigma}^i ) \in \op{T}(\op{C}^{\infty}(\Sigma,\, M)
\times \Sigma),
$$
and so by Lemma \ref{trivial:lem}
$$
\op{d}(\op{ev})({v_i}_{ (f,\,x )}) =\op{d}\!f_{x} (v_{\Sigma}^i) \in \op{d} \! f_{x}(\op{T}_{x}\Sigma).
$$

On the other hand
we have that
\begin{eqnarray} 
    \int_{\{f\} \times \Sigma} \iota_{( \ell_1' \wedge \ell_2' )} \op{ev}^{\ast} (
    \omega \wedge \omega ) \nonumber
 \end{eqnarray}
 is equal to
 $$
 \int_{\{f\} \times \Sigma} \iota_{( \ell_1 \wedge
    \ell_2 )} \op{ev}^{\ast} ( \omega \wedge \omega ) + \int_{\{f\} \times
    \Sigma} \iota_{( v_1 \wedge \ell_2 )} \op{ev}^{\ast} ( \omega \wedge
    \omega ) + \int_{\{f\} \times \Sigma} \iota_{(\ell_1 \wedge v_2 )}
    \op{ev}^{\ast} ( \omega \wedge \omega ) + \int_{\{f\} \times \Sigma}
    \iota_{( v_1 \wedge v_2 )} \op{ev}^{\ast} ( \omega \wedge \omega).
 $$
 
To complete the proof it is enough to show that the last three terms vanish. We
will show that
  their integrands are identically zero when restricted to $\op{T}( \{f\} \times
  \Sigma$). Let $z_1,\,z_2$ be any pair of vectors in
  $\op{T}(\{f\} \times \Sigma)$.  Then 
  \begin{eqnarray}
   \Big( \iota_{( v_1 \wedge l_2 )} \op{ev}^{\ast} ( \omega \wedge \omega) \Big)
   ( z_1, \, z_2 ) & = & \Big( \iota_{(\op{d}(\op{ev})( v_1 ) \wedge \op{d}(\op{ev})(\ell_2 )
    )} (\omega \wedge \omega) \Big) ( \op{d}(\op{ev})( z_1 ), \, \op{d}(\op{ev})( z_2 ) )
\nonumber\\
    & = & \Big( \iota_{(\op{d}\!f ( v_{\Sigma}^1 ) \wedge \op{d}(\op{ev})(\ell_2 ) )}
(\omega
    \wedge \omega) \Big) (\op{d}\!f ( z_1 ), \, \op{d}\!f ( z_2 ) ). \label{ex1}
      \end{eqnarray}

Similarly we obtain that
 \begin{eqnarray}
       \Big(\iota_{( \ell_1 \wedge v_2 )} \op{ev}^{\ast} ( \omega \wedge \omega )\Big)
       ( z_1, \,z_2 ) & = & \Big(\iota_{( \op{d}(\op{ev})(\ell_1 ) \wedge \op{d} (\op{ev})(
       v_2 ) )} (\omega \wedge \omega) \Big) (\op{d}(\op{ev})( z_1 ),\, \op{d}(\op{ev})(
       z_2 ) ) \nonumber \\
       & = & \Big( \iota_{( \op{d}(\op{ev})(\ell_1 ) \wedge \op{d}\!f ( v_{\Sigma}^2 ) )}
       (\omega \wedge \omega) \Big)( \op{d}\!f ( z_1 ), \,\op{d}\!f ( z_2 ) ), \label{ex2}
     \end{eqnarray}
and that
       \begin{eqnarray}
       \Big(\iota_{( v_1 \wedge v_2 )} \op{ev}^{\ast} ( \omega \wedge \omega )\Big)
       ( z_1, \,z_2 ) & = & \Big( \iota_{( \op{d}(\op{ev})( {v}_1 ) \wedge
       \op{d}(\op{ev})( v_2 ) )} (\omega \wedge \omega) \Big)(\op{d}(\op{ev})( z_1 ),\,
       \op{d}(\op{ev})( z_2 ) )\nonumber \\
       & = & \Big( \iota_{(\op{d}\!f ( v_{\Sigma}^2 ) \wedge \op{d} \!f ( v_{\Sigma}^2))}
       (\omega \wedge \omega) \Big)(\op{d}\! f ( z_1 ),\, \op{d}\! f ( z_2 ) ). \label{ex3}
  \end{eqnarray}

The terms    
    $$
     \iota_{(\op{d}\!f ( v_{\Sigma}^1 ) \wedge \op{d}(\op{ev})(\ell_2 ) )}
(\omega
    \wedge \omega),\,\,\,\,
     \iota_{( \op{d}(\op{ev})(\ell_1 ) \wedge \op{d}\!f ( v_{\Sigma}^2 ) )}
       (\omega \wedge \omega) ,\,\,\,\,
       \iota_{(\op{d}\!f ( v_{\Sigma}^2 ) \wedge \op{d} \!f ( v_{\Sigma}^2))}
       (\omega \wedge \omega)
  $$
   each vanish when restricted to the $2$\--dimensional subspace $\op{d}\!f_{x} (\op{T}_{x} \Sigma )   \subset \op{T}_{f(x)} M$, by Lemma  \ref{lemma3}.
  It follows that the right hand side in expressions (\ref{ex1}), (\ref{ex2}) and
  (\ref{ex3}) identically vanishes. Hence
   \begin{eqnarray} \label{equality}
    \int_{\{f\} \times \Sigma} \iota_{(\ell'_1 \wedge \ell'_2 )} \op{ev}^{\ast} ( \omega
  \wedge \omega)= \int_{\{f\} \times \Sigma} \iota_{(\ell_1 \wedge
  \ell_2 )} \op{ev}^{\ast} ( \omega \wedge \omega). \nonumber
      \end{eqnarray}
\end{proof}

\begin{lemma} \label{lemma3}
   Let $W$ be a vector space, and let $\alpha$ be a
  $4$\--form: $\alpha \colon \bigwedge^4 W \rightarrow \R.$ 
  Let $V \subset W$ be a
  subspace of dimension $\leq 2$. Then
 $\Big(\iota_{( v \wedge w )} \alpha\Big) |_V = 0$
  for all $v \in V$, $w \in W$.  
 \end{lemma}
This is the case since  any three vectors in
$V$ are linearly
  dependent.

\subsection*{The form is closed}

\begin{proof}[Proof of the closedness part of Theorem \ref{thclonon}]

The $2$\--form ${\omega}_{\op{C}^{\infty}(\Sigma,\, M)}$
is closed if and only if for any
  two surfaces $R_1$ and $R_2$ in
   $\op{C}^{\infty}(\Sigma,\, M)$, 
   that are homologous relative
  to a common boundary $\partial R$,
  \begin{eqnarray} \label{xx}
  \int_{R_1} {\omega}_{\op{C}^{\infty}(\Sigma,\, M)}
  = \int_{R_2} 
  {\omega}_{\op{C}^{\infty}(\Sigma,\, M)}.
\end{eqnarray}
 
 Now (by considering the liftings that are zero along $\Sigma$), 
  $$
  \int_{R_i} {\omega}_{\op{C}^{\infty}(\Sigma,\, M)}
   = \int_{R_i \times \{x\}} \left( \int_{\{f\} \times
     \Sigma} \op{ev}^{\ast} ( \omega \wedge \omega )  \right) =
     \int_{R_i \times \Sigma} \op{ev}^{\ast} (\omega \wedge \omega).
 $$
Since $R_1 \times \Sigma$ is homologous to $R_2 \times \Sigma$ relative to the
  boundary $\partial R \times \Sigma$, and $\op{ev}^{\ast} ( \omega \wedge
  \omega)$ is closed, we have that:
  $$
  \int_{R_1 \times \Sigma} \op{ev}^{\ast} ( \omega \wedge \omega )
      = \int_{R_2 \times \Sigma} \op{ev}^{\ast} ( \omega \wedge
     \omega ).
     $$
 Therefore we get \eqref{xx}.
\end{proof}

\subsection*{Directions of degeneracy}

\begin{proof}[Proof of the domain of non-degeneracy in Theorem \ref{thclonon}]
\textup{\,}
\\
{\bf Case 1}.
{\em  Suppose that $f \colon \Sigma \to M$, and $\tau \colon \Sigma \to f^{\ast}\op{T}M$ is everywhere tangent to $f ( \Sigma )$.}
 \\
  We claim that ${\omega_{\op{C}^{\infty}(\Sigma,\, M)}}
   _{f}(\tau,\,\cdot)=0$.  To see this, lift $\tau$ to a vector field $\ell=(\tau,\,0)$ along 
  $\op{C}^{\infty}(\Sigma,\, M)
  \times \Sigma$; let $\tau_2 \in \op{T}_{f} (\op{C}^{\infty}(\Sigma,\, M))
  $ and $\ell_2$ a lifting of $\tau_2$. We show that the integrand $\iota_{\ell \wedge \ell_2} \op{ev}^{*} (\omega \wedge \omega)$ vanishes when restricted to $\op{T}(\{f\} \times \Sigma)$. 
 Indeed, for $z_1,\,z_2 \in \op{T}_x(\{f\} \times \Sigma),$
$$ \iota_{\ell \wedge \ell_2} \op{ev}^{*}(\omega \wedge \omega)_x(z_1,\,z_2) = \iota_{\tau \wedge \op{d} \op{ev} (\ell_2)}(\omega \wedge \omega)_{f(x)}(\op{d}\! f (z_1), \,\op{d}\! f(z_2)).$$ So it is enough to show that 
$$ \iota_{\tau \wedge \op{d} \op{ev} (\ell_2)}(\omega \wedge \omega) | _{\op{d}\!f(\op{T}_{x} \Sigma)}$$ vanishes. This follows from Lemma \ref{lemma3}, since, by assumption $\tau(x) \in \op{d}\!f_{x}(\op{T}_{x}\Sigma)$ and $\, \op{d}\!f_{x}(\op{T}_{x}\Sigma) \subset \op{T}_{f(x)}M$ is a $2$\--dimensional subspace.
\\
\\
{\bf Case 2}.  {\em Suppose that $f \in {\mathcal{S}_{\op{i}}(\Sigma)}$, and $\tau \in \op{T}_{f}({\mathcal{S}_{\op{i}}(\Sigma)})$ is not tangent to $f ( \Sigma )$ at $x \in \Sigma$.}
\\
Let  $\tau_{\perp} \in (f^{\ast} \op{T}\! M)_{x}$ denote (a representative of) the orthogonal projection of $\tau$ to the normal bundle to  $\op{d}\!f (\op{T}\Sigma)$. In particular, $\tau_{\perp}(x) \neq 0$.
  Let $\tau_1$ be a vector in $( f^{\ast} (\op{T}\! M))_x$ such that 
  $$\omega(\tau_{\perp} (x),\,\tau_1)=\omega_{f(x)}(\tau_{\perp} (x),\,\tau_1) > 0,$$ 
  and $\tau_1$ is symplectically orthogonal to $\op{d}\!f_x ( \op{T}_x \Sigma )$.

  Now extend $\tau_1$ to a section $\tau_1  \colon
  \Sigma \rightarrow f^{\ast}(\op{T}\! M)$ such that $\omega( \tau_{\perp}( y), \, \tau_1(y) ) > 0$ and $\tau_1(y)$ is symplectically orthogonal to $\op{d}\!f_y ( \op{T}_y \Sigma )$ for $y$ in a small
  neighborhood of $x$, and vanishing outside it.
   Let $\ell$ and $\ell_1$ be liftings of $\tau_{\perp}$ and $\tau_1$ that are zero along $\Sigma$. 
   We claim that   ${\omega_{\mathcal{S}_{\op{i}}(\Sigma)}}_{f}(\tau_{\perp},\tau_1)
 \neq 0$.

Notice that, in general for vectors $\mu_1, \, \mu_2 \colon
  \Sigma \rightarrow f^{\ast}(\op{T}\! M)$ and their zero-liftings $k_1^0=(\mu_1,0), \, k_2^0=(\mu_2,0)$, we have that for $z_1,\,z_2 \in \op{T}(\{f\} \times \Sigma)$,
  \begin{eqnarray}
  \iota_{k_1^0\wedge k_2^0} (\op{ev}^*(\omega \wedge \omega))(z_1, z_2)
  &=&\Big(\iota_{(\op{d}(\op{ev})(k_1^0) \wedge \op{d}(\op{ev})(k_2^0))}(\omega \wedge \omega)\Big)
(\op{d}(\op{ev})(z_1),\, \op{d}(\op{ev})(z_2)) \nonumber \\
&=&\Big(\iota_{\mu_1\wedge \mu_2}(\omega \wedge \omega)\Big)(\op{d}\!f(z_1),\, \op{d}\!f(z_2)) \nonumber \\
&=& 2 \, \omega(\mu_1,\, \mu_2)\, \omega(\op{d}\!f(z_1),\, \op{d}\!f(z_2)) + 2 \omega (\mu_1, \, \op{d}\!f(\cdot)) \wedge \omega(\mu_2, \, \op{d}\!f(\cdot))(z_1,\,z_2). \nonumber
  \end{eqnarray}
   In particular, and since $\tau_1(y)$ is  symplectically orthogonal to $\op{d}\!f_y( \op{T}_y \Sigma )$ for every $y \in \Sigma$ (hence the second summand in the last term vanishes), we get that $$\int_{\{f\} \times \Sigma}  \iota_{( \ell \wedge \ell_1 )} \op{ev}^{\ast} ( \omega \wedge \omega )=
  \int_{\Sigma} \omega(\tau_{\perp},\tau_1) f^{\ast}\omega \neq 0,$$
  where the last inequality follows from the choice of $\tau_1$, and the fact that, by definition of ${\mathcal{S}_{\op{i}}(\Sigma)}$, the form
 $f^{\ast}\omega$ is a symplectic $2$-form on a surface, hence an area-form.
                 
   Since, (by Case 1),
  ${\omega_{\mathcal{S}_{\op{i}}(\Sigma)}} _{f}( \tau, \,\tau_1 ) ={\omega_{\mathcal{S}_{\op{i}}(\Sigma)}}_{f} ( \tau_{\perp}, \,\tau_1 )$, 
  we deduce  that
  $$
  {\omega_{\mathcal{S}_{\op{i}}(\Sigma)}}_{f} ( \tau_, \,\tau_1 )  \neq 0.
  $$
 \end{proof}

\subsection*{Compatible almost complex structures on $\mathcal{S}_{\op{i}}(\Sigma)$}

An \emph{almost complex structure} on a manifold $M$ is an automorphism of
the tangent bundle, $$J \colon \op{T}\!M \to \op{T}\!M,$$ such that $J^2 = -\op{Id}$.
The pair $(M,\,J)$
is called an \emph{almost complex manifold}. 

An almost complex structure is
\emph{integrable} if it is induced from a complex manifold structure. In
dimension two any almost complex manifold is integrable (see, e.g.,
\cite[Th.~4.16]{MS}). In higher dimensions this is not true
\cite{calabi}.

An almost complex
structure $J$ on $M$ is 
\emph{tamed} by a symplectic form $\omega$ if  $\omega_{x}(v,\,Jv) >0$ for all non-zero $v \in T_x M$. If, in addition, $\omega_{x}(Jv,\,Jw)=\omega_{x}(v\,,w)$ for all $v, \, w \in T_x M$, we say that $J$ is \emph{$\omega$-compatible}. 
The space $\mathcal{J}(M,\omega)$ of $\omega$-compatible almost complex structures is non-empty and contractible, in particular path-connected \cite[Prop.~4.1]{MS}.

\begin{definition} \label{defj}
  Let $J$ be an almost complex structure on $M$. We define a map
  $$
  \tilde{J}:=J_{\mathcal{S}_{\op{i}}(\Sigma)} : \op{T}\!\mathcal{S}_{\op{i}}(\Sigma) \rightarrow 
  \op{T}\!\mathcal{S}_{\op{i}}(\Sigma)
  $$
   as follows:
for $\tau \colon \Sigma \rightarrow f^{\ast} (\op{T}\!M)$, the vector $\tilde{J}(\tau)$
  is the section
  $
  \hat{J} \circ \tau,
  $
 where $\hat{J}$ is the map defined by the commutative diagram
   \begin{eqnarray} 
\xymatrix{ \ar @{} [dr] |{\circlearrowleft}
f^*(\op{T}\!M)  \ar[r]^{\hat{J}}      \ar[d]  &  f^*(\op{T}\!M)
                  \ar[d]   \\
                  \op{T}\!M \ar[r]^J   &       \op{T}\!M } \nonumber
\end{eqnarray}     
 \end{definition}

Since $\mathcal{S}_{\op{i}}(\Sigma)$ is an open subset of $\op{C}^{\infty}(\Sigma,\, M)$, we get that $\op{T}\!\mathcal{S}_{\op{i}}(\Sigma)$ is indeed closed under $\tilde{J}$. Due to the properties of the almost complex structure $J$, the map $\tilde{J}$ is an automorphism and ${\tilde{J}}^{2}=  -\op{Id}$.
\begin{claim} 
 Let $J$ be an almost complex structure on $M$.
 Then $\tilde{J}$ is an almost complex structure on $\mathcal{S}_{\op{i}}(\Sigma)$. 
\end{claim}

A smooth ($\op{C}^{\infty}$) curve $f \colon \Sigma \rightarrow M$ from a compact  Riemann surface $(\Sigma,\, j)$  to an almost complex manifold $(M,J)$ is called
$(j,\,J)$\--\emph{holomorphic} if the differential $\op{d}\!f$ is a complex linear map
between the fibers $\op{T}_{p}(\Sigma)\rightarrow \op{T}_{f(p)}(M)$ for all
$p \in
\Sigma$, i.e. 
\[\op{d}\!f_{p}\circ {j}_{p} = {J}_{f(p)}\circ
\op{d}\!f_{p}.\]
 (When $j$ is clear from the context, we call $f$ a 
\emph{$J$\--holomorphic curve}.)

\smallskip
\smallskip
\noindent

Fix $\Sigma=(\Sigma,\,j)$.
\begin{claim} \label{jclaim}
 Let $J$ be an almost complex structure on $M$.
Assume that $f \colon \Sigma \to M$ is $J$-holomorphic. Then, at $x \in \Sigma$, 
\begin{enumerate}
   \item[\textup{1.}] if $\tau(x) \in \op{d}\!f_{x}( \op{T}_{x}\Sigma )$ then $J_{f(x)} (\tau_x) \in \op{d}\!f_{x}( \op{T}_{x}\Sigma )$; 
   \item[\textup{2.}] if $J$ is $\omega$-compatible and $\phi_x$ is symplectically orthogonal with respect to  $\op{d}\!f_{x}( \op{T}_{x}\Sigma )$, then so is $J_{f(x)} (\phi_x)$.
 \end{enumerate}  
\end{claim} 

\begin{proof}
\begin{enumerate}
\item By assumption $\tau_x=\op{d}\! f_{x}(\alpha)$ for $\alpha \in \op{T}_x \Sigma$. Hence, since $f$ is $J$-holomorphic,
 $$J_{f(x)} (\tau_x)=J_{f(x)} (\op{d}\!f_{x}(\alpha))=\op{d}\! f_{x}(j_{x} \alpha).$$ 
\item By the previous item, $J_{f(x)} (\op{d}\!f_{x}( \op{T}_{x}\Sigma )) \subseteq \op{d}\!f_{x}( \op{T}_{x}\Sigma )$, hence, since $J^2=-\id$, $$J_{f(x)}(\op{d}\!f_{x}( \op{T}_{x}\Sigma ))=\op{d}\!f_{x}( \op{T}_{x}\Sigma ).$$
Let $\tau_x \in \op{d}\!f_{x}( \op{T}_{x}\Sigma )$, then there exists $\tau'_x \in \op{d}\!f_{x}( \op{T}_{x}\Sigma )$ such that $\tau_x=J_{f(x)}(\tau'_x)$. By assumption, $\omega(\phi_x, \tau'_x)=0$. Thus $$\omega(J_{f(x)}(\phi_x),\tau_x)=
\omega(J_{f(x)}(\phi_x),J_{f(x)}(\tau'_x))=\omega(\phi_x,\tau'_x)=0.$$

\end{enumerate}
\end{proof}

\begin{corollary} \label{corj}
Let $J$ an $\omega$-tamed almost complex structure on $M$. Assume that $f \colon \Sigma \to M$ is $J$-holomorphic.  Then, at $x \in \Sigma$, for $W_x=\op{d}\!f_{x}( \op{T}_{x}\Sigma )$, $W_x \cap {W_x}^{\omega}=\{0\}$.
\end{corollary}

\begin{proof}
By part (1) of Claim \ref{jclaim}, if $ v \in W_x$, then $J(v) \in W_x$; since $J$ is $\omega$-tamed, if $v \neq 0$,
$\omega(v,J(v))>0$, hence $0 \neq v \in W_x$ is not in  ${W_x}^{\omega}$.
\end{proof}

Since $\dim {W_x}^{\omega}=2 \dim M - \dim W_x$, this implies the following corollary.
\begin{corollary} \label{corj2}
In the assumptions and notations of Corollary \ref{corj},
$$\op{T}_{f(x)}M=W_x+{W_x}^{\omega}.$$
\end{corollary}

Recall that if a bundle $E \to B$ equals the direct sum of sub-bundles $E_1 \to B$ and $E_2 \to B$, then the space of sections of $E$ equals the direct sum of the space of sections of $E_1$ and the space of sections of $E_2$. Thus, Corollary \ref{corj2} implies the following corollary. 
\begin{corollary} \label{corj3}
Let $J$ an $\omega$-tamed almost complex structure on $M$. Assume that $f \colon \Sigma \to M$ is $J$-holomorphic. 
 Then every $\mu \in \op{T}_f({\mS_{\op{i}}(\Sigma))}$ can be uniquely decomposed as $$\mu=\mu'+\mu'',$$ where 
 $\mu'$ is everywhere tangent to $f(\Sigma)$ (i.e. $\mu' (x) \in  \op{d}\!f_{x} (\op{T}_{x} \Sigma)$ at every $x \in \Sigma$), and $\mu''(x)$ is symplectically orthogonal to  $\op{d}\!f_{x}( \op{T}_{x}\Sigma )$ at every $x \in \Sigma$.
\end{corollary}

\begin{proposition} \label{integcomp}
 Let $J$ be an almost complex structure on $M$, and $f \colon \Sigma \to M$ a $J$-holomorphic map. 
\begin{enumerate}
\item[{\rm (1)}] Assume that $J$ is $\omega$-tamed. Then  
for $\tau$ that is \emph{not} everywhere tangent to $f(\Sigma)$, $ {\omega_{\mS_{\op{i}}(\Sigma)}}_{f} ( \tau, \tilde{J}(\tau) ) > 0.$ 
\item[{\rm (2)}] Assume that $J$ is $\omega$-compatible. Then $\tilde{J}$ is compatible with   $ {\omega_{\mS_{\op{i}}(\Sigma)}}$.
\end{enumerate}
\end{proposition}

\begin{proof}
\begin{enumerate}
\item[(1)] By Corollary \ref{corj3}, part (1) of Claim \ref{jclaim}, and Theorem \ref{thclonon}, it is enough to show that $ {\omega_{\mS_{\op{i}}(\Sigma)}}_{f} (\phi, \tilde{J}(\phi)) \neq 0$ for $\phi$ such that  $\phi(x)$ is symplectically orthogonal to  $\op{d}\!f_{x}( \op{T}_{x}\Sigma )$ at every $x \in \Sigma$. Indeed, for such $\phi$,
for $z_1,\,z_2 \in \op{T}(\{f\} \times \Sigma)$,
  \begin{eqnarray}
&&  \iota_{(\phi,0) \wedge (\tilde{J}(\phi),0)} (\op{ev}^*(\omega \wedge \omega))(z_1, z_2)  \nonumber \\
  &=&\Big(\iota_{(\op{d}(\op{ev})(\phi,0) \wedge \op{d}(\op{ev})(\tilde{J}(\phi),0))}(\omega \wedge \omega)\Big)
(\op{d}(\op{ev})(z_1),\, \op{d}(\op{ev})(z_2)) \nonumber \\
&=&\Big(\iota_{\phi \wedge \tilde{J}(\phi)}(\omega \wedge \omega)\Big)(\op{d}\!f(z_1),\, \op{d}\!f(z_2)) \nonumber \\
&=& 2 \, \omega(\phi,\, \tilde{J}(\phi))\, \omega(\op{d}\!f(z_1),\, \op{d}\!f(z_2)) + 2 \omega (\phi, \, \op{d}\!f(\cdot)) \wedge \omega(\tilde{J}(\phi), \, \op{d}\!f(\cdot))(z_1,z_2). \nonumber
  \end{eqnarray}
  
  By assumption on $\phi$, the second summand in the last term equals zero, so $$\int_{\{f\} \times \Sigma}  \iota_{( (\phi,0) \wedge (\tilde{J}(\phi),0))} \op{ev}^{\ast} ( \omega \wedge \omega )=
  \int_{\Sigma} \omega(\phi,\tilde{J}(\phi)) f^{\ast}\omega \neq 0,$$
  where the last inequality follows from the fact that $J$ is $\omega$-tamed and $f^{\ast}\omega$ is a symplectic $2$-form on a surface, hence an area-form.

\item[(2)] 
By Corollary \ref{corj3}, part (1) of Claim \ref{jclaim}, and Theorem \ref{thclonon}, it is enough to show that 
$$ {\omega_{\mS_{\op{i}}(\Sigma)}}_{f} (\phi_1, \phi_2) =  {\omega_{\mS_{\op{i}}(\Sigma)}}_{f} (\tilde{J}(\phi_1),\tilde{J}(\phi_2))$$
for $\phi_1, \, \phi_2$ such that  $\phi_{i}(x)$ is symplectically orthogonal to  $\op{d}\!f_{x}( \op{T}_{x}\Sigma )$ at every $x \in \Sigma$.
Indeed the above calculation shows that for such $\phi_i$,  
 \begin{eqnarray}
&& \iota_{(\tilde{J}(\phi_1),0) \wedge (\tilde{J}(\phi_2)),0)} (\op{ev}^*(\omega \wedge \omega))(z_1, z_2) \nonumber \\
 &=& 2 \, \omega(\tilde{J}(\phi_1),\, \tilde{J}(\phi_2))\, \omega(\op{d}\!f(z_1),\, \op{d}\!f(z_2)) + 2 \omega (\tilde{J}(\phi_1), \, \op{d}\!f(\cdot)) \wedge \omega(\tilde{J}(\phi_2), \, \op{d}\!f(\cdot))(z_1,z_2). \nonumber
 \end{eqnarray}
 Since $J$ is $\omega$-compatible, the first summand in the last term equals  $2 \, \omega(\phi_1,\, \phi_2)\, \omega(\op{d}\!f(z_1),\, \op{d}\!f(z_2))$.
 By assumption on $\phi_i$, and since $J$ is $\omega$-compatible, we deduce by part (2) of Claim \ref{jclaim} that  $\tilde{J}(\phi_{i})(x)$ is symplectically orthogonal to  $\op{d}\!f_{x}( \op{T}_{x}\Sigma )$ at every $x \in \Sigma$, therefore the second summand in the last term equals zero. Thus  
\begin{eqnarray}
 \iota_{(\tilde{J}(\phi_1),0) \wedge (\tilde{J}(\phi_2)),0)} (\op{ev}^*(\omega \wedge \omega))(z_1, z_2) 
 &=& 2 \, \omega(\phi_1,\, \phi_2)\, \omega(\op{d}\!f(z_1),\, \op{d}\!f(z_2)) \nonumber \\
 &=& \iota_{(\phi_1,0) \wedge (\phi_2,0)} (\op{ev}^*(\omega \wedge \omega))(z_1, z_2).  \nonumber
 \end{eqnarray}
 (The last equality follows again from the assumption on $\phi_i$.)

\end{enumerate}

\end{proof}

%%%%%%%%%%%%%%%%%%%%%%%%%%%%%%

\section{Moduli spaces of $J$\--holomorphic curves} \label{mod}
\label{2.2}

We review here the ingredients of the theory of J-holomorphic curves
that
we need in this study,  as we learned from  \cite{gromovcurves} and \cite{MS2}, and deduce a few facts that we use later in the paper.

Let $(M,\,J)$ be an almost complex $2n$\--manifold. Fix a compact  Riemann
surface $\Sigma=(\Sigma,\,j)$. 

\subsection*{Simple $J$--holomorphic curves}

 We say that a $J$-holomorphic curve is 
 \emph{simple} if it is not the composite of a holomorphic branched covering map
  $(\Sigma,\,j) \to (\Sigma',\,j')$ of degree greater than one with a $J$-holomorphic map 
  $(\Sigma',\,j') \to (M,\,J)$.

\subsection*{The operators ${  \bar{\partial}_J}$ and ${  \op{D}_f}$}

The $J$\--holomorphic maps from
$(\Sigma,\,j)$ to $(M,\,J)$ are the maps satisfying $\bar{\partial}_{J}(f)=0$,
where $$\bar{\partial}_J(f) := \frac{1}{2}\, (\op{d}\!f + J \circ \op{d}\!f \circ
j).$$ Let $A \in
\op{H}_{2}(M,\,\Z)$ be a homology class. The $\bar{\partial}_J$ operator
defines a section $S \colon \fB \to \E$,
\begin{eqnarray} \label{SBE}
S(f):=(f,\,\bar{\partial}_J(f)),
\end{eqnarray}
where $\fB \subset \op{C}^{\infty}(\Sigma,\, M)$
denotes the space of all smooth maps $f \colon \Sigma \to M$ that represent
the homology class $A$, and the bundle $\E \to \fB$ is the infinite
dimensional vector bundle whose fiber at $f$ is the space 
$\E_f: =\Omega^{0,\,1}(\Sigma,\, f^{*}(\op{T}\!M))$
of smooth $J$\--antilinear 1\--forms on
$\Sigma$ with values in $f^{*}(\op{T}\!M)$.

Given a $J$-holomorphic map $f \colon (\Sigma,j) \to (M,J)$ in the class $A$, the \emph{vertical differentiation} of the section $S$ at $f$
\begin{eqnarray} \label{vertdif:def}
\op{D}_f \colon \Omega^0(\Sigma,\, f^*(\op{T}\!M)) \to \Omega^{0,1}(\Sigma,
f^*(\op{T}\!M)) \nonumber
\end{eqnarray}
is the composition of the differential
$
\op{D}\!S(f) \colon \op{T}_f\mathcal{B}
\to \op{T}_{(f,\,0)}\mathcal{E}
$
with the projection map
$\pi_f \colon \op{T}_{(f,\,0)}\mathcal{E} = \op{T}_{f} \mathcal{B} \oplus
\mathcal{E}_f \to \mathcal{E}_f.$

 Let $(M,\omega)$ be  a symplectic manifold and $J$ an $\omega$-compatible almost complex structure on $M$. % \in \mathcal{J}(M,\omega)$. 
 The operator
 $\op{D}_f$  can be expressed 
as a sum
\begin{eqnarray} \label{Dbarf}
\op{D}_f(\xi)=\op{D}^f_1(\xi)+\op{D}^f_2(\xi) \nonumber,
\end{eqnarray}
where  $\op{D}^f_1(\xi)$ is complex linear (meaning that
$\op{D}^f_1(J\xi)=J\, \op{D}^f_1(\xi)$)
and 
\begin{eqnarray} \label{Dbarf2}
\op{D}^f_2(\xi)=\frac{1}{4}\, \op{N}_J(\xi,\, \partial_J(f)).
\end{eqnarray}
is complex antilinear (meaning that
$\op{D}^f_2(J\xi)=-J\,\op{D}^f_2(\xi)$).
See \cite[Remark 3.1.2]{MS2}.
Recall that the Nijenhuis tensor $\op{N}_J$ is defined by
\begin{eqnarray} \label{njtensor}
\op{N}_J(X,\,Y):=[JX,\,JY]-J\,[JX,\,Y]-J[X,\,JY]-[X,\,Y] \nonumber
\end{eqnarray}
where $X,\,Y\colon M \to \op{T}\!M$ are vector fields on $M$.
It is a result of Newlander\--Nirenberg that $J$ is integrable
if and only if $\op{N}_J = 0$, c.f. \cite{NN} or \cite[Thm. 4.12]{MS}.

\begin{lemma} \label{jdf}
If $\tau \in \Omega^0(\Sigma, \, f^*(\op{T}M))$ is such that
$\op{D}_f(\tau)=0$, then
\begin{eqnarray} \label{Df2tau}
\op{D}_f(J \, \tau)=-2J\,\op{D}^f_2(\tau).
\end{eqnarray}
\end{lemma}

\begin{proof}
If $\op{D}_f(\tau)=0$, then
$\op{D}^f_1(\tau)=-\op{D}^f_2(\tau)$, and therefore
we have that
\begin{eqnarray}
\op{D}_f(J \, \tau)=J\,\op{D}^f_1(\tau)-J\,\op{D}^f_2(\tau)=
-J\,\op{D}^f_2(\tau)-J\,\op{D}^f_2(\tau)=-2J\,\op{D}^f_2(\tau).
\nonumber
\end{eqnarray}
\end{proof}

\subsection*{Moduli spaces of simple $J$\--holomorphic curves}

\begin{definition}
The \emph{moduli space of $J$\--holomorphic maps from $(\Sigma,\, j)$ to $(M,\,J)$} in the class $A$
is the set
$
\M(A, \, \Sigma,\,J)
$
given as the
intersection of the zero set of the section $S$ in (\ref{SBE}), and the open set of 
all simple
 maps
$\Sigma \to M$ which represent the class $A$. 
\end{definition}

Explicitly,
\begin{eqnarray} \label{moduli:def1}
\M(A,\ \Sigma,\,J):=
\{f \in \op{C}^{\infty}(\Sigma,\, M) \ \mid \ f
\textup{ is a simple}\,\, (j,\,J)\textup{\--holomorphic map in }A\}. \nonumber
\end{eqnarray}

The tangent space to $
\M(A, \, \Sigma,\,J)
$ at $f$ is the zero set of $\op{D}_f$.

\begin{lemma} \label{jinteg}
If $J$ is an integrable almost complex structure on $M$, and $f \colon \Sigma \to M$ is a $J$-holomorphic map, then for $\tau \in T_{f}\M(A, \,\Sigma,\,J)$, the vector $J \circ \tau$ is also in  $T_{f}\M(A, \,\Sigma,\,J)$.
\end{lemma}

 \begin{proof}
Since $J$ is integrable $\op{N}_J=0$, hence by (\ref{Dbarf2}) and
(\ref{Df2tau}),
${\op{D}}_f ( J \circ \tau )= 0$.
\end{proof}

\subsection*{The universal moduli space}

Let $\J:=\J(M,\omega)$. 
\begin{definition}
The \emph{universal moduli space of holomorphic maps} from
$\Sigma$ to $M$, in the class  $A$, is the set
\begin{eqnarray} \label{moduli:def2}
\M(A,\,\Sigma,\,\J):=\{(f,\,J) \ \mid \ J
\in \J, \ f \in \M(A,\ \Sigma,\,J) \}. \nonumber
\end{eqnarray}
\end{definition}

\subsection*{Regular value and regular path}

We have the following consequences of the 
Sard-Smale theorem,
the infinite dimensional inverse mapping theorem,
and the ellipticity of the Cauchy-Riemann equations.

\begin{lemma} \label{consard}
\begin{enumerate}
\item[{\rm (a)}]
The universal moduli space $\M(A,\,\Sigma,\,\J)$ and the space $\J$ 
are Fr\'echet manifolds.

\item[{\rm (b)}]

Consider the projection map 
$$p_{A} \colon \M(A,\,\Sigma, \, \mathcal{J}) \to \J$$
for any $J$ in the set of $\omega$-compatible $J$'s that are regular for the map 
$p_A$. The moduli space
$\M(A, \, \Sigma,\,J)$
is a smooth manifold of dimension $2\op{c}_1(A)+n(2-2g)$

\item[{\rm (c)}]
The set   of $\omega$-compatible $J$'s that are regular for the map 
$p_A$ is of the second category in~$\J$.

\item[{\rm (d)}]
If $(u,\,J)$ is a regular value for $p_A$, then for any neighborhood $U$ 
of $(u,\,J)$ in $\M(A,\Sigma, \, \J)$, its image, $p_A(U)$, contains a neighborhood 
of $J$ in $\J$.

\item[{\rm (e)}]
Let $J_0, \, J_1 \in  \J$, assume that $J_0, \, J_1$ are regular for $p_A$. If a path $\lambda \to J_{\lambda}$ is transversal to $p_A$, then 
$$\W(A,\, \Sigma,\, \{J_{\lambda}\}_{\lambda}):=\{(\lambda,f) \, | \, 0 \leq \lambda \leq 1, \, f \in \M(A,\, \Sigma,\, 
J_{\lambda})\}$$
is a smooth oriented manifold with boundary $\partial \W(A,\, \Sigma,\,\{J_{\lambda}\}_{\lambda})=
\M(A,\, \Sigma,\, J_0) \cup \M(A,\, \Sigma,\, J_1).$
The boundary orientation agrees with the orientation of $\M(A,\,\Sigma,\, J_1)$ and is opposite to the orientation of $\M(A,\, \Sigma,\, J_0)$.

\item[{\rm (f)}]
Let $\{ J_t \}_{t \in [0,1]}$
be a 
$\op{C}^1$ simple path in $\J$ whose endpoints are regular values
for $p_A$.  Then there exists a 
$C^1$ 
perturbation
$\{ \tilde{J}_t \}$ of $J_t$ with the same endpoints
which is transversal to $p_A$.

\end{enumerate}
\end{lemma}

For items (a), (d) and (f) see, e.g., \cite{KKP} 
and \cite{KKP0} where they are derived
using the results of~\cite[chapter 3]{MS2}. For items (b) and (c) see \cite[Theorem 3.1.5]{MS2}. For item (e) see \cite[Theorem 3.1.7]{MS2}.

%%%%%%%%%%%%%%%%%%%%%%%%%%%

\section{Symplectic structure on moduli spaces of simple immersed $J$\--holomorphic curves} \label{modim}

Consider the moduli subspaces of 
$\mathcal{S}_{\op{i}}(\Sigma)$:
\begin{eqnarray}
\mathcal{M}_{\op{i}}(A,\, \Sigma,\,J):=\mathcal{M}(A,\, \Sigma,\,J) \cap \mathcal{S}_{\op{i}}(\Sigma)
 \,\,\,\,\,\textup{and}\,\,\,\,\,
\mathcal{M}_{\op{i}}(A,\, \Sigma,\,\mathcal{J}):=
\mathcal{M}(A,\, \Sigma,\,\mathcal{J}) \cap \mathcal{S}_{\op{i}}(\Sigma). \nonumber
\end{eqnarray}

\begin{remark}
For $J \in \J$, 
if $f$ is a $J$\--holomorphic curve and an immersion then 
$f^{*}\omega$ is a symplectic form. 
Closedness is immediate since $\omega$ is closed
and the differential commutes with the pullback.
The fact that  $f^{*}\omega$ is not degenerate is by 
the following argument. For $u \neq 0$, 
$$
f^{*}\omega(u,\,j(u))=\omega(\op{d}\!f(u), \,
\op{d}\!f(j(u)))=\omega(\op{d}\!f(u),\,J\op{d}\!f(u)) >0, 
$$
the last strict inequality is since 
$J$ is $\omega$\--compatible and $\op{d}\!f(u) \neq 0$ (since 
$u \neq 0$ and $\op{d}\!f$ is an immersion). 
Thus $\mathcal{M}_{\op{i}}(A,\, \Sigma,\, \J)$ is the subset of immersions in $\mathcal{M}(A,\, \Sigma,\,\J)$. 
\end{remark}

\begin{lemma} \label{mifr}
The moduli space
$\mathcal{M}_{\op{i}}(A,\, \Sigma,\,\mathcal{J})$ is a Fr\'echet submanifold of 
the F\'rechet manifold $\mathcal{M}(A,\, \Sigma,\, \mathcal{J})$.

If $J \in \J$ is regular for ${p_A}|_{\mathcal{M}_{\op{i}}(A,\, \Sigma,\,\mathcal{J})}$, the moduli space
$\mathcal{M}_{\op{i}}(A,\, \Sigma,\,J)$ is a finite dimensional manifold.
Moreover, the statements of Lemma \ref{consard}
hold true for 
$\mathcal{M}_{\op{i}}(A,\, \Sigma,\J)$, $\mathcal{M}_{\op{i}}(A,\, \Sigma,\,J)$, and  ${p_A}|_{\mathcal{M}_{\op{i}}(A,\, \Sigma,\,\mathcal{J})}$.

\end{lemma}

\begin{corollary} \label{corcob}
Assume that $S$ is an open and path-connected subset of $\J$, and that $S \subset  \J_{\scriptop{reg}}(A)$. Then for every $J_0, \, J_1 \in S$ there is an 
oriented cobordism between the manifolds $\M_{\scriptop{i}}(A,\,\Sigma,\,J_0)$ and $\M_{\scriptop{i}}(A,\,\Sigma,\,J_1)$ that is the $p_A$-preimage of a $p_A$-transversal path from $J_0$ to $J_1$ in $\J$. 
\end{corollary}

\subsection*{The form restricted to the moduli space of simple immersed $J$-holomorphic curves is symplectic up to reparametrizations when $J$ is integrable or close to integrable on a path}

As a result of Lemma \ref{jinteg} and Proposition \ref{integcomp}, we get the following proposition.
\begin{proposition} \label{propintsym}
  If $J$ is an integrable almost complex structure on $M$ that is compatible
  with $\omega$ and regular for $A$, then $\omega_{\mS_{\op{i}}(\Sigma)}$ restricted to
  $\M_{\op{i}}(A,\,\Sigma,\,J) $ is symplectic up to reparametrizations.
\end{proposition}

\begin{remark} \label{remregst}
For examples of regular integrable compatible almost complex structures, we look at K\"ahler manifolds whose automorphism groups act transitively. By \cite[Proposition 7.4.3]{MS2}, if $(M,\,\omega_0,\,J_0)$ is a compact K\"ahler manifold and $G$ is a Lie group that acts transitively on $M$ by holomorphic diffeomorphisms, then $J_0$ is regular for every $A \in \op{H}_2(M,\,\Z)$. This applies, e.g., when $M = \C P^n$, $\omega_0$ the Fubini-Study form, $J_0$ the standard complex structure on $\C P^n$, and $G$ is the automorphism group $\PSL(n+1)$, 
\end{remark}

\begin{remark}
Notice that if $J$ is not integrable, then $\op{T}_{f}\M_{\op{i}}(A,\,\Sigma,\,J) $ is not necessarily closed under $\tilde{J}$, so non-degeneracy is harder to witness.
\end{remark}

\begin{remark} \label{remquo}
If the compact Riemann surface $\Sigma=(\Sigma,j)$ is of genus $0$, its group of automorphisms $\Aut(\Sigma,j)$ is $\PSL(2,\,\C)$ and its action on   
 $\M(A,\,\Sigma,\,\J) $ is proper, assuming that $0 \neq A \in \op{H}_2(M,\,\Z)$ (see, e.g., Lemma  3.1 in \cite{thesis}). 
 If $\Sigma$ is of genus $1$, it is a torus ${\C}/(\Z+\alpha \Z)$,  where the imaginary part
of  $\alpha$  is nonzero;  the group  $\Aut(\Sigma,\,j)$ contains the torus
itself (as left translations). When the genus is $\geq 2$, the automorphism group is finite, by Hurwitz's automorphism Theorem. 
So the action of $\Aut(\Sigma,j)$ on  $\M_{\scriptop{i}}(A,\,\Sigma,\,J)$ is proper;
therefore if the form $\omega_{\mathcal{S}_{\op{i}}(\Sigma)}$ on $\M_{\op{i}}(A,\,\Sigma,\,J)$ is symplectic up to reparametrizations, it
descends to a symplectic $2$-form 
${\widetilde{\omega}}_{\mathcal{S}_{\op{i}}(\Sigma)}$ on the quotient space  ${\widetilde{\M}}_{\scriptop{i}}(A,\,\Sigma,\,J)$.
\end{remark}

\begin{noTitle} \label{nt1}
A class $A \in \op{H}_2(M,\, \Z)$ is \emph{$J$-indecomposable} if it does not split as a sum 
$A_1 + \ldots + A_k$ of classes all of which can be represented by non-constant $J$-holomorphic curves. 
The class $A$ is called \emph{indecomposable} if it is 
$J$-indecomposable for all $\omega$-compatible $J$. Notice that if $A$ cannot be written 
as a sum $A = A_1 + A_2$ where $A_i \in \op{H}_2(M,\, \Z)$
and $\int_{A_i}{\omega} > 0$,   then it is indecomposable.

Assume that $M$ is compact.
If $A$ is indecomposable, then Gromov's compactness theorem \cite[1.5.B.]{gromovcurves} implies that  
if $J_n$ converges in $\J$, then, modulo parametrizations,  every sequence $(f_n,J_n)$ in $\M(A,\, \Sigma , \, \J)$ has a $(C^{\infty}-)$convergent subsequence.
Therefore the map  $p_A \colon \M(A,\, \Sigma ,\, \J) /  \Aut(\Sigma) \to \J$ induced by $p_A$ is proper; in particular every quotient space ${\widetilde{\mathcal{M}}}(A,\,\Sigma,\, J)$ is compact.

Similarly, if $A$ is $J$-indecomposable for all $J$ in a set $S \subset \J$, then the map ${p_A}^{-1}S / \Aut(\Sigma) \to S$ induced by $p_A$ is proper; in particular, its image is closed in $S$.
\end{noTitle}

We can now prove Corollary \ref{vol}.

\begin{proof}[Proof of Corollary \ref{vol}]

By part (f) of Lemma \ref{consard},  a path in $\J$ between $J_0, \, J_1 \in S$ can be perturbed to a $p_A$-transversal path with the same endpoints. By part (e) of Lemma \ref{consard}, the $p_A$-preimage of the path,  
$\W(A,\Sigma,\{J_{\lambda}\}_{\lambda})$, is a smooth oriented manifold with boundary $\partial \W(A,\Sigma,\{J_{\lambda}\}_{\lambda})=\M(A,\Sigma,J_0) \cup \M(A,\Sigma,J_1),$
and the boundary orientation agrees with the orientation of $\M(A,\Sigma,J_1)$ and is opposite to the orientation of $\M(A,\Sigma,J_0)$.
By Gromov's compactness theorem the quotient $\widetilde{\W}(A,\Sigma,\{J_{\lambda}\}_{\lambda})$, (i.e., the preimage of the path under the map  $p_A \colon \M(A,\, \Sigma ,\, \J) /  \Aut(\Sigma) \to \J$ induced by $p_A$), is compact (see \S \ref{nt1}).  

Thus, using Theorem \ref{thclonon}, the integration of ${\wedge^{n} {\widetilde{\omega}}_{{\op{C}}^{\infty}(\Sigma,\, M)}}$ along ${\widetilde{\mathcal{M}}(A,\,\Sigma,\, J_i) }$ (for $i=0,1$) is well defined, and by Stokes' Theorem (and the fact that ${\widetilde{\omega}}_{{\op{C}}^{\infty}(\Sigma,\, M)}$ is closed), we get that  
\begin{eqnarray}
0&=&\int_{\widetilde{\W}(A,\,\Sigma,\, \{J_{\lambda}\}_{\lambda})}{\op{d}(\wedge^{n} {\widetilde{\omega}}_{{\op{C}}^{\infty}(\Sigma,\, M)})} 
=
\int_{\partial \widetilde{\W}(A,\,\Sigma,\, \{J_{\lambda}\}_{\lambda})}(\wedge^{n} {\widetilde{\omega}}_{{\op{C}}^{\infty}(\Sigma,\, M)}) \nonumber \\
&=&
\int_{\widetilde{\mathcal{M}}(A,\,\Sigma,\, J_0) }{\wedge^{n} {\widetilde{\omega}}_{{\op{C}}^{\infty}(\Sigma,\, M)}}-\int_{\widetilde{\mathcal{M}}(A,\,\Sigma,\, J_1) }{\wedge^{n} {\widetilde{\omega}}_{{\op{C}}^{\infty}(\Sigma,\, M)}}.\nonumber
\end{eqnarray}

If there is an integrable $J_* \in S$ such that ${\mathcal{M}_{\scriptop{i}}(A,\,\Sigma,\, J_*)} \neq \emptyset$, then, by  Proposition \ref{propintsym},  $\int_{{\widetilde{\mathcal{M}}}_{\scriptop{i}}(A,\,\Sigma,\, J_*)} {\wedge^{n} {\widetilde{\omega}}_{\mS_{\scriptop{i}}(\Sigma)}} \neq 0$ hence, by the above, $\int_{{\widetilde{\mathcal{M}}}(A,\,\Sigma,\, J)} {\wedge^{n} {\widetilde{\omega}}_{{\op{C}}^{\infty}(\Sigma,\, M)}} \neq 0$ for every $J \in S$.

\end{proof}

\begin{lemma} \label{open}
Let $(M,\,\omega)$ be a symplectic manifold, $A \in \op{H}_2(M,\,\Z)$, and $\Sigma=(\Sigma,j)$ a compact Riemann surface. 
Assume that $J_*$ is an integrable and $A$-regular $\omega$-compatible almost complex structure. 
Let $$\{J_{\lambda}\}_{0 \leq \lambda \leq 1},$$ with $J_0=J_{*}$, be a path in $\J=\J(M,\omega)$ that is $A$-regular, i.e., the path is transversal to $p_A$ and $J_{\lambda} \in  \J_{\scriptop{reg}}(A)$ for all $0 \leq \lambda \leq 1$. Assume that the map 
\begin{equation} \label{papr}
 {p_A}^{-1}( \{J_{\lambda}\}_{0 \leq \lambda \leq 1}) / \Aut(\Sigma,j)  \to \{J_{\lambda}\}_{0 \leq \lambda \leq 1} 
\end {equation}
that is induced from $p_A$ is proper.

Then there is an open neighbourhood $I$ of $J_{*}$ in  $\{J_{\lambda}\}_{0 \leq \lambda \leq 1}$, such
that for every $J \in I$, the restriction of the form
 $\omega_{\mS_{\op{i}}(\Sigma)}$ to the moduli space
  $\M_{\scriptop{i}}(A,\,\Sigma,\,J) $ is symplectic up to reparametrizations. 
\end{lemma}

\begin{proof}
By part (e) of Lemma \ref{consard} and Lemma \ref{mifr}, the universal moduli space over the path, $${p_A}^{-1}(\{J_{\lambda}\})=\W(A,\Sigma,\{J_{\lambda}\}_{\lambda})$$ is a finite dimensional manifold; for each $0 \leq \lambda \leq 1$, since $J_\lambda \in  \J_{\scriptop{reg}}(A)$, the moduli space $\M_{\scriptop{i}}(A,\,\Sigma,J_{\lambda})$ is a finite dimensional manifold. Moreover, since \eqref{papr} is proper,  we get that 
the quotient $\widetilde{\W}(A,\Sigma,\{J_{\lambda}\}_{\lambda})$ is compact, and each of the quotient spaces ${\widetilde{\M}}_{\scriptop{i}}(A,\,\Sigma,\,J)$ is compact. By Proposition \ref{propintsym}, the  restriction of the form
 ${\widetilde{\omega}}_{\mathcal{S}_{\op{i}}(\Sigma)}$ to  ${\widetilde{\M}}_{\scriptop{i}}(A,\,\Sigma,\,J_{*})$ is non-degenerate.

Since non\--degeneracy is an open condition, for every $f \in {\widetilde{\M}}_{\scriptop{i}}(A,\,\Sigma,\,J_{*})$ there is an open neighbourhood  
$$f \in V_f  \subseteq \widetilde{\W}(A,\Sigma,\{J_{\lambda}\}_{\lambda}),$$ such that for all $(h,J) \in V_f$, for every 
$\tau_h \in \op{T}_{h}{\widetilde{\M}}_{\scriptop{i}}(A,\,\Sigma,\,J)$,  the map ${\widetilde{\omega}}_{\mS_{\op{i}}(\Sigma)}(\tau_h,\cdot) $ on  $\op{T}_{h}{\widetilde{\M}}_{\scriptop{i}}(A,\,\Sigma,\,J)$ is not equal to zero. ($V_f$  should be small enough such that $h$ and 
$ \op{T}_{h}{\widetilde{\M}}_{\scriptop{i}}(A,\,\Sigma,\,J)$ are close to $f$ and  to $\op{T}_{f}\M_{\scriptop{i}}(A,\,\Sigma,\,J_{*})$, respectively.)

Let $$U_f = V_f \cap {\widetilde{\M}}_{\scriptop{i}}(A,\,\Sigma,\,J_{*}).$$ 
Since ${\widetilde{\M}}_{\scriptop{i}}(A,\,\Sigma,\,J_{*})$ is compact, it is covered by finitely many $U_f$-s. Let $V$ be the intersection of the corresponding $V_f$-s, then $V$ is an open neighbourhood of ${\widetilde{\M}}_{\scriptop{i}}(A,\,\Sigma,J_{*})$ in  ${\widetilde{\W}}(A,\Sigma,\{J_{\lambda}\}_{\lambda})$.  By properness of the map (\ref{papr}), there is an open neighborhood $J_* \in I \subset \{J_{\lambda}\}_{0 \leq \lambda \leq 1}$, such that the preimage of $I$ under \eqref{papr} is contained in  $V.$
This completes the proof.

\end{proof}

\subsection*{Compatible almost complex structures on moduli spaces of simple immersed J-holomorphic curves}

We apply the following well known fact, c.f. \cite[Prop. 12.6]{AC}

\begin{proposition} \label{alm}
Let $(N, \, \omega)$ be a symplectic manifold, and $g$ a Riemannian
metric on $N$. Then  there exists a canonical almost complex structure
$J$ on $N$ which is $\omega$-compatible. The structure $J$ equals $\sqrt{A A^{*}}^{-1}A$ where 
$A \colon \op{T}\!N \to \op{T}\!N$ is such that $\omega(u,\,v)=g(Au,\,v)$. 
\end{proposition}

\begin{corollary} \label{coralm}
Let $J \in {\mathcal{J}(M,\omega)}_{\reg}$ such that  the form ${\widetilde{\omega}}_{\mathcal{S}_{\op{i}}(\Sigma)} $  on ${\widetilde{\mathcal{M}}}_{\op{i}}(A,\,\Sigma,\, J)$  is symplectic. Let $g$ be the metric induced on the quotient space by
 the $\op{L}^2$\--norm $$
\sqrt{ {\widetilde{\omega}}_{\mathcal{S}_{\op{i}}(\Sigma)}  (\cdot,J \cdot)}.
$$ Then, by Proposition \ref{alm}, we obtain
a canonical  almost complex structure on 
${\widetilde{\mathcal{M}}}_{\op{i}}(A,\,\Sigma,\,J)$  that is compatible with the form.
\end{corollary}

%%%%%%%%%%%%%%%%%%%

\section{Examples} \label{examples}
In this section $\Sigma$ is $\C P^1(=S^2)$ with the standard complex structure.

\begin{example}
In case $(M,\omega)=(\C P^n, \omega_{FS})$ and $A=L$, the homology class of $\C P^1$ that generates $\op{H}_2(M,\,\Z)$, the class $A$ is indecomposable. By \cite[Proposition 7.4.3]{MS2}, the standard  complex structure on $\C P^n$ is regular (see Remark \ref{remregst}). Thus, by  Corollary \ref{vol}, for every $A$-regular almost complex structure $J \in \J$ there is a $J$-holomorphic sphere in $A$.

In general, if $M$ is compact, $\op{H}_2(M,\,\Z)$ is of dimension one, and $A$ is the generator of  $\op{H}_2(M,\,\Z)$ of minimal (symplectic) area, then $A$ is indecomposable.  If $(M,\,\omega_0,\,J_0)$ is a compact K\"ahler manifold whose automorphism group acts transitively, then $J_0$ is regular for every $C \in \op{H}_2(M,\,\Z)$  \cite[Proposition 7.4.3]{MS2}. Thus if $A$ is represented by a $J_0$-holomorphic sphere, then, by Corollary \ref{vol}, for every $p_A$-regular almost complex structure $J \in \J$ there is a $J$-holomorphic sphere in $A$. As an example of such a manifold, consider the Grassmannian $M=G/ P$ of oriented 2-planes in $\R^n$: $G=\SO(n)$, $P=S_1 \times \SO(n-2)$. We thank Yael Karshon for suggesting the Grassmannian example.

\end{example}

\begin{noTitle} \label{nt2}
Let $(M,\omega)$ be a symplectic four-manifold. By the adjunction inequality, in a four-dimensional manifold, 
if $A  \in \op{H}_2(M,\, \Z)$ is represented by a simple $J$-holomorphic sphere $f$, then
$$ A \cdot A - \op{c}_1(A) + 2 \geq 0,$$
with equality if and only if $f$ is an embedding; see~\cite[Cor.~E.1.7]{MS2}.
Thus, if  there is $J' \in \J$ such that $A=[u]$ for an embedded $J'$-holomorphic sphere $u \colon \C P^1 \to M$, then for every $(f, \, J) \in \M(A,\C P^1,\J)$, the sphere $f$ is an embedding; in particular,  $$\M(A,\C P^1,\J)=\mathcal{M}_{\op{i}}(A,\,\C P^1, \, \J).$$
 The existence of such  $J' \in \J(M,\omega)$ 
 is guaranteed when 
  $A$ is represented by an embedded symplectic sphere (see \cite[Section 2.6]{MS}).

The Hofer-Lizan-Sikorav regularity criterion asserts that in a four-dimensional manifold, if $f$ is an 
immersed $J$-holomorphic sphere, then $(f,\, J)$ is a regular point 
for the projection $p_{[f]}$ if and only if 
$c_{1}([f])\geq 1$,  \cite{HLS}.)

Therefore, if $A \in \op{H}_2(M,\, \Z)$ is such that $\op{c}_1(A) \geq 1$, and $A$ is represented by an embedded $J'$-holomorphic sphere for some almost complex structure $J'$ on a four-manifold $M$, then every $(f, \,J) \in  \M(A,\C P^1,\J)(=\mathcal{M}_{\op{i}}(A,\,\C P^1, \, \J))$ is a regular point for $p_A$, thus, by part (d) of Lemma \ref{consard}, the image of $p_A$ is open in $\J$.
\end{noTitle}

As a result of \ref{nt1} and \ref{nt2} we get the following lemma.
\begin{lemma} \label{lemc1}
Let $(M,\omega)$ be a compact symplectic four-manifold, and  $A \in \op{H}_2(M,\, \Z)$ with 
$\op{c}_1(A) \geq 1$. Assume that $S \subset \J(M,\omega)$ is such that:
\begin{itemize}
\item  $A$ is $J$-indecomposable for all $J \in S$;
\item $A$ is represented by an embedded $J'$-holomorphic sphere for some $J \in S$;
\item $S$ is connected. 
\end{itemize}
Then the map $p_A \colon {p_A}^{-1}(S) \to  S$ is  onto, and its image is open and closed, thus $S=p_A({p_A}^{-1}(S)) \subseteq  \mathcal{J}_{\scriptop{reg}}(A)$.
\end{lemma}

In particular if $S$ satisfies the assumptions of Lemma \ref{lemc1} and $J _* \in S$ is integrable, then Lemma \ref{open} applies to every $p_A$-transversal path $\{J_{\lambda}\}_{0 \leq \lambda \leq 1}$ in $S$, with $J_0=J_*$.  If 
$S \subset \J_{\reg}(A)$ is path-connected and open then by Lemma \ref{consard}, there is a $p_A$-transversal path in $S$ between every two elements of $S$. 
We give here three examples of $(M,\omega,A)$ and $S$, in all of which $S$ satisfies the assumptions of Lemma \ref{lemc1}, and $S$ is path-connected and open;
in addition, there is an integrable $J_* \in S$ such that ${\mathcal{M}_{\scriptop{i}}(A,\,\Sigma,\, J_*)} \neq \emptyset$.
Consequently, Corollary \ref{vol} implies the existence of a $J$-holomorphic sphere in $A$ for every $J$ in the outlined sets $S$; in all of the examples $S$ is dense in $\J$ thus we get the existence of a $J$-holomorphic sphere in $A$ for a generic $J$.

\begin{example}
In case $(M,\,\omega)=(\C P^2,\, \omega_{FS})$ and $A=L$, the homology class of $\C P^1$ that generates
$\op{H}_2(M,\,\Z)$, the assumptions of Lemma \ref{lemc1} are satisfied for the set $S=\J(M,\,\omega).$  %This set is of course open, dense, and path-connected in $\J(M,\,\omega)$.  

The same holds in case $(M,\,\omega)=(S^2 \times S^2, \, \tau \oplus \tau),$ where $\tau$ is the rotation 
invariant area form on $S^2$ (with total area equal to $1$), and $A=[S^2 \times \{\pt\}]$.

 In each of these cases, there is a standard integrable compatible complex structures in $S$, and $A$ is represented by a sphere that is holomorphic for the standard structure.
\end{example}

\begin{example}
Consider $(S^2 \times S^2,\, (1+\lambda) \tau \oplus \tau)$. (When $\lambda >0$, this symplectic manifold  has a compatible almost complex structure $J$ for which there is a non-regular sphere, namely the antidiagonal $\overline{D}=\{(s,-s) \in S^2 \times S^2\}$; see \cite[Example 3.3.6]{MS2}.) 
By  Abreu \cite[Sec. 1.2]{abreu}, for $0 <\lambda \leq 1$,
the subset $\J_{\lambda}^b$  of $(1+\lambda) \tau \oplus \tau$-compatible almost complex structures for which the class $[\overline{D}]$ is represented by a (unique embedded) J-holomorphic sphere is a non-empty, closed, codimension $2$ submanifold of $\J=\J(S^2 \times S^2,\, (1+\lambda) \tau \oplus \tau)$. Abreu also shows that for all $J$ in the complement of $\J_{\lambda}^b$, the class $A=[S^2 \times \{\pt\}]$ is $J$-indecomposable. Thus, the assumptions of Lemma \ref{lemc1} are satisfied for $S=\J \smallsetminus \J_{\lambda}^b$, and $S$  is also open dense and path-connected in $\J$.
Notice that the standard split compatible complex structure $j \oplus j$ is in $S$, and that the class  $A=[S^2 \times \{\pt\}]$ is represented by (a $2$-parameter family of) embedded spheres $S^2 \times \{s\}$ that are holomorphic for $j \oplus j$.
 Therefore, by Corollary \ref{vol}, there is a $J$-holomorphic sphere in $A$ for every $J$ in the complement of $\J_{\lambda}^b$.

Abreu \cite[Theorem 1.8]{abreu} shows that   $\J \smallsetminus \J_{\lambda}^b$ equals the space $\J_{\lambda}^g$ of $J$ for which the homology class $[S^2 \times \{\pt\}]$  is represented by an embedded $J$-holomorphic sphere.
\end{example}

\begin{example}
Let $(M,\,\omega)$ be a compact symplectic four-manifold and $A \in \op{H}_2(M,\, \Z)$ a homology class that can be represented by an embedded symplectic sphere and  such that $\op{c}_1(A)=1$. Consider the set of  
$J \in \J(M,\,\omega)$ for which there is a non-constant $J$-holomorphic sphere in a homology class $H \in \op{H}_2(M,\, \Z)$ with $\op{c}_1(H)<1$ and $\omega(H)<\omega(A)$, and denote its complement by $U_A$. Then, by definition, $A$ is $J$-indecomposable for every $J$ in $S=U_A$.    
The set $S=U_A$ is an open dense and path-connected subset of $\J(M,\,\omega)$, see, e.g., \cite[App. A]{KKP}. 

In case  $(M,\,\omega)$ is obtained by a sequence of blow ups from $(\C P^2,\,\omega_{FS})$ or from 
$(S^2 \times S^2, \,\tau \oplus \tau)$ and $A$ is the homology class of one of the blow ups, there is an
 integrable structure $J_*$ in $S=U_A$, and $A$ is represented by a $J_*$-holomorphic sphere. Therefore, by Corollary \ref{vol},  there is a $J$-holomorphic sphere in $A$ for every $J$ in $U_A$.   
In \cite[App. A]{KKP},  it is shown that for every $J \in U_A$, the class $A$ is represented by an embedded $J$-holomorphic sphere.
\end{example}

{\bf Acknowledgements}. We are grateful 
to D. Auroux, H. Hofer, Y. Karshon, D. McDuff, and R. Melrose
for helpful
discussions. J. Coffey and L. Kessler were affiliated with the Courant
Institute when this
project started. 
The project was partially funded by NSF postdoctoral grants 
and a VIGRE grant DMS-9983190.

\noindent

\bigskip\noindent
Liat Kessler
\\
Massachusetts Institute of Technology,
Mathematics Department\\
Cambridge MA 02139 USA\\
{\em E\--mail}: kessler@math.mit.edu\\

\smallskip\noindent
Joseph Coffey\\
{\em E\--mail}:joe.coffey@gmail.com\\

\smallskip\noindent
\'Alvaro Pelayo\\
University of California\---Berkeley,
Mathematics Department\\
970 Evans Hall $\#$ 3840\\
Berkeley, CA 94720-3840 USA\\
{\em E\--mail}: {apelayo@math.berkeley.edu}

\end{document}